\documentclass[12pt,reqno]{amsart}

\usepackage{hyperref}

\usepackage[T1]{fontenc}
\usepackage[latin1]{inputenc}
\usepackage{latexsym}
\usepackage[arrow, matrix, curve]{xy}
\usepackage[dvips]{graphicx}
\usepackage{epsfig}
\usepackage{bm}
\usepackage{bbm}
\usepackage{accents}
\usepackage{amssymb}
\usepackage{amsmath}
\usepackage{verbatim}
\usepackage{amsthm}
\usepackage{enumerate}
\usepackage{mathrsfs}
\usepackage[hmargin=3cm,vmargin=3cm]{geometry}
\newtheorem{theorem}{Theorem}
\newtheorem{lemma}[theorem]{Lemma}
\newtheorem{corollary}[theorem]{Corollary}

\theoremstyle{definition}

\theoremstyle{remark}

\numberwithin{equation}{section}

\def\ba{{\mathbf a}}

\def\bn{{\mathbf n}}

\def\bw{{\mathbf w}}
\def\bx{{\mathbf x}}
\def\by{{\mathbf y}}
\def\bz{{\mathbf z}}

\def\calH{{\mathcal H}}

\def\calX{{\mathcal X}}

\def\N{\mathbb N}

\def\Q{\mathbb Q}
\def\R{\mathbb R}

\def\T{\mathbb T}

\def\Z{\mathbb Z}

\def\frc{{\mathfrak c}}

\def\alp{{\alpha}}

\def\eps{\varepsilon}

\def\sig{{\sigma}}

\def\balp{{\boldsymbol \alpha}}

\def\bgam{\boldsymbol \gamma}

\def\d{{\partial}}
\def\bzer{\boldsymbol 0}

\def\le{\leqslant} \def\ge{\geqslant}

\def\d{{\,{\rm d}}}

\usepackage{hyperref}

\begin{document}
\title{On the inhomogeneous Vinogradov system}
\author[Julia Brandes]{Julia Brandes}
\address{Mathematical Sciences, University of Gothenburg and Chalmers Institute of 
Technology, 412 96 G\"oteborg, Sweden}
\email{brjulia@chalmers.se}

\author[Kevin Hughes]{Kevin Hughes}
\address{School of Mathematics, University of Bristol, Fry Building, Woodland Road, Bristol BS8 1UG, UK}
\email{khughes.math@gmail.com}


\begin{abstract}
	We show that the system of equations 
	\begin{align*}
		\sum_{i=1}^s (x_i^j-y_i^j) = a_j \qquad (1 \le j \le k)
	\end{align*}	
	has appreciably fewer solutions in the subcritical range $s < k(k+1)/2$ than its homogeneous counterpart, provided that $a_\ell \neq 0$ for some $\ell \le k-1$. Our methods use Vinogradov's mean value theorem in combination with a shifting argument. 
\end{abstract}

\maketitle

\section{Introduction}\label{S1}

Sitting squarely at the interface of harmonic analysis and analytic number theory, systems of diagonal equations have long attracted a significant amount of attention. In particular, thanks to the resolution of the main conjecture associated with Vinogradov's mean value theorem by Wooley \cite{W:EC3} and Bourgain, Demeter and Guth \cite{BDG} (see also \cite{W:NEC}), we now have essentially sharp upper bounds for the number of integer solutions of systems of equations of Vinogradov type. 

For integral $s, k$ denote by  $J_{s,k}(X)$ the number of solutions $\bx, \by \in \Z^s \cap [-X,X]^s$ satisfying 
\begin{align}\label{V-sys}
	\sum_{i=1}^s (x_i^j-y_i^j) = 0 \qquad (1 \le j \le k),
\end{align}
although we remark for future reference that by transitioning to exponential sums the quantity $J_{s,k}(X)$ may be defined also for non-integer positive $s$, see \eqref{V-def}. 
In this notation, \cite[Theorem~1.1]{BDG} (see also \cite[Corollary~1.3]{W:NEC}) shows that for all positive numbers $s$ and all $k \in \N$ one has
\begin{align}\label{VMVT}
	J_{s,k}(X) \ll X^{\eps} (X^s + X^{2s-k(k+1)/2}). 
\end{align}

The second term in \eqref{VMVT}, dominant when $s > k(k+1)/2$, conforms with the circle method heuristic which states that a system of equations in $n$ variables and of total degree $K$ should have roughly $\sim c X^{n-K}$ solutions over the integers of size at most $X$, where the factor $c$ encodes the density of solutions of the underlying system over all completions of $\Q$. Indeed, in the case of \eqref{V-sys} such an asymptotic formula is readily established for $s > k(k+1)/2$ by a classical application of the Hardy--Littlewood method (e.g. \cite[Chapter 7]{V:HL}). Moreover, these methods are robust with respect to small modifications to the system \eqref{V-sys} such as non-trivial coefficients or a non-vanishing expression on the right hand side of \eqref{V-sys}. 

When $s$ is smaller than $k(k+1)/2$, the circle method heuristic fails for \eqref{V-sys}, as we have to allow for the existence of diagonal solutions in which  the variables $(x_1, x_2, \ldots, x_s)$ are a permutation of $(y_1, y_2,\ldots, y_s)$. The contribution of these solutions amounts to $s!X^s$, so it is natural to expect an asymptotic formula of the shape 
\begin{align}\label{VMVT-paucity}
	J_{s,k}(X) \sim s!X^s
\end{align} 
to hold in the subcritical range. Adventurous souls might even go as far as to conjecture that the non-diagonal contribution to the mean value $J_{s,k}(X)$ should be no larger than $O(X^\eps(1 + X^{2s-k(k+1)/2}))$, based on an expectation that the circle method heuristic should continue to be valid in that setting. Somewhat frustratingly, however, we are able to establish \eqref{VMVT-paucity} only in the range $s \le k+1$ (Vaughan and Wooley \cite[Theorem~1]{VW}), which falls far short of the conjectured range $s < k(k+1)/2$. 

The purpose of this note is to study an inhomogeneous version of the system \eqref{V-sys} that does not allow for diagonal contributions. For $s, k \in \N$ and $a_1, \ldots, a_k \in \Z$ write $J_{s,k}(X;\ba)$ for the number of $\bx, \by \in \Z^s \cap [-X,X]^s$ satisfying the system of equations 
\begin{align}\label{IV-sys}
	\sum_{i=1}^s (x_i^j-y_i^j) = a_j \qquad (1 \le j \le k).
\end{align}
Systems of this or similar types have recently arisen in a range of different contexts, including for instance \cite{CKMS,DHV,MS}. It follows easily by a standard application from the triangle inequality that 
\begin{align}\label{triv-bd}
	J_{s,k}(X;\ba) \le J_{s,k}(X) \ll X^{s+\eps}
\end{align} 
in the subcritical range. We show that for most $\ba$ we can do better. 

\begin{theorem}\label{MT}
	Suppose that $k \ge 2$ and $\ba \in \Z^k \setminus \{\bzer\}$. Let $\ell\in \{1, \ldots, k\}$ be the smallest integer for which $a_{\ell} \neq 0$. Then for any integer $s < \frac12 k(k+1)$ and any $\eps > 0$ we have 
	\begin{align}\label{IVMVT}
		J_{s,k}(X; \ba)\ll X^{s-1/2+\eps} + X^{s- \eta_{s,k}(\ell)  +\eps},
	\end{align} 
	where 
	\begin{align*}
		\eta_{s,k}(\ell) = \frac{(k-\ell)(k-\ell+1)}{2} \left(1-\frac{2s}{k(k+1)}\right).
	\end{align*}
\end{theorem}
A brief computation reveals that the first term in \eqref{IVMVT} dominates for $s=\frac12 k(k+1) - 1$ as soon as $a_{\ell} \neq 0$ for some $\ell \le k-\frac12(\sqrt{2k^2+2k+1}-1)$. Hence we have the following simple consequence of Theorem~\ref{MT}.
\begin{corollary}
	Suppose that $k \ge 2$ and assume that $a_\ell \neq 0$ for some 
	\begin{align*}
		\ell \le k-\tfrac12(\sqrt{2k^2+2k+1}-1).
	\end{align*}
	Then for any integer $s \le \frac12 k(k+1)-1$ and any $\eps > 0$ we have 
	\begin{align*}
		J_{s,k}(X; \ba)\ll X^{s-1/2+\eps}.
	\end{align*} 
\end{corollary}
Here, one can compute that 
\begin{align*}
	k-\tfrac12(\sqrt{2k^2+2k+1}-1) &= \left(1-\frac{1}{\sqrt 2}\right) \left(k + \frac12 \right) + O(1/k)\\&\approx 0.292...\cdot k + 0.146... + O(1/k). 
\end{align*}

It may be worth pointing out that we do not expect to obtain power savings over the bound \eqref{triv-bd} when $s$ is at or beyond the critical point and $|a_j| \le 2sX^j$ for $1 \le j \le k$, since at that point the contribution from a major arcs analysis will be of size $\gg X^{2s-k(k+1)}$, see e.g. \cite[Chapter~7]{V:HL} for details. In this sense, the range $s \le \frac12k(k+1) -1$ is optimal. 

The proof of Theorem \ref{MT} relies crucially on the absence of translation invariance in the system \eqref{IV-sys}. For this reason, our result becomes gradually weaker as the system acquires larger translation-invariant subsystems, up to the point when $\ba = (0, \ldots, 0, a_k)^t$, where we have $\ell=k$ and $\eta_{s,k}(k)=0$. At this point, the entire system becomes translation-invariant, so we fail to make any progress beyond the trivial bound \eqref{triv-bd}. 

Our strategy transfers to several other settings, but in the interest of a slick presentation we have opted to focus our attention on the inhomogeneous Vinogradov system, since this seems to be by far the most relevant example. However, we discuss some possible extensions of our result as well as some of its limitations in the final section.

\bigskip 
\textbf{Notation.}  We denote the unit torus by $\T = \R /\Z$. Throughout, the letter $\eps$ will be used to denote an arbitrary positive number, and we adopt the convention that whenever it appears in a statement, we assert that the statement holds for all $\eps>0$. Finally, we take $X$ to be a large positive number which, just like the implicit constants in the notations of Landau and Vinogradov, is permitted to depend at most on $s$, $k$ and $\eps$.

\section{Proof of the main theorem}
In our proof we employ a strategy inspired by work of Wooley \cite{Woo2015} to understand systems with incomplete translation-invariant structure.  Our first step is to bound $J_{s,k}(X;\ba)$ in terms of a different quantity which will be easier to handle. Define polynomials $p_j$ by setting 
\begin{align}\label{poly}
	p_j(h) = \sum_{m=1}^{j} \binom{j}{m} a_{m}h^{j-m} \qquad (1 \le j \le k).
\end{align}
Write then $H_{s,k}(X;\ba)$ for the number of $\bz, \bw \in \Z^s \cap [-2X,2X]^s$ and $h \in \Z \cap [-X,X]$ satisfying the system of equations 
\begin{align}\label{s+a}
	\sum_{i=1}^s (z_i^j-w_i^j) = p_j(h) \qquad (1 \le j \le k). 
\end{align}
Then we have the following. 
\begin{lemma}\label{J<H}
	We have 
	\begin{align*}
		J_{s,k}(X;\ba) \ll X^{-1} H_{s,k}(X;\ba).
	\end{align*}
\end{lemma}
\begin{proof}
	Suppose that $(\bx, \by)$ is a solution to \eqref{IV-sys}, and fix a parameter $h \in \Z$. It then follows from the binomial theorem that $(\bx,\by)$ also is a solution of the shifted system 
	\begin{align*}
		\sum_{i=1}^s \left((x_i+h)^j-(y_i+h)^j\right) = p_j(h) \qquad (1 \le j \le k). 
	\end{align*}
	If now $|h| \le X$, then the number of such solutions $\bx, \by \in \Z^s \cap [-X,X]^s$ is certainly no larger than the number of $\bz, \bw \in \Z^s \cap [-2X,2X]^s$ satisfying \eqref{s+a} with that particular value of $h$. The desired conclusion now follows upon summing over all $|h| \le X$. 
\end{proof}

Next, we write $H_{s,k}(X;\ba)$ in terms of mean values over exponential sums. Set 
\begin{align*}
	f_k(\balp;X) = \sum_{|x| \le X} e(\alp_1 x + \ldots + \alp_k x^k) 
\end{align*}
and 
\begin{align}\label{g-def}
	g_k(\balp;X) = \sum_{|x| \le X} e(\alp_1 p_1(h) + \ldots + \alp_k p_k(h)),
\end{align}
and recall that in this notation we have 
\begin{align}\label{V-def}
	J_{s,k}(X) = \int_{\T^k} |f_k(\balp;X)|^{2s} \d \balp.
\end{align}
Similarly, it follows from standard orthogonality relations that we may write 
\begin{align}\label{H-int}
	H_{s,k}(X;\ba) = \int_{\T^k} |f_k(\balp;2X)|^{2s} g_k(-\balp;X) \d \balp.
\end{align}

From the definition of the $p_j$ in \eqref{poly}, we see that $g_k(\balp;X)$ is an exponential sum of degree at most $k-1$; however, its degree may be lower depending on the values of the coefficients $a_1, a_2, \ldots, a_k$. In particular, recalling that $\ell \in \{1, \ldots, k\}$ denotes the smallest integer for which $a_{\ell} \neq 0$, we discern from \eqref{poly} that
\begin{align}\label{deg-p}
	\deg p_j = \max\{0, j-\ell\} \qquad (1 \le j \le k).
\end{align}
Thus, $g_k(\balp;X)$ is an exponential sum containing $k-\ell$ integer polynomials with degrees $1, \ldots,k-\ell$, respectively. 

We now apply H\"older's inequality to \eqref{H-int}, and find that 
\begin{align}\label{H\"older}
	H_{s,k}(X;\ba) \le \left(\int_{\T^k} |f_k(\balp;2X)|^{k(k+1)} \d \balp\right)^{\frac{2s}{k(k+1)}} \left(\int_{\T^k} |g_k(\balp;X)|^{2\sig}\d \balp\right)^{1/(2\sig)},
\end{align}
where we wrote 
\begin{align}\label{sig}
	\sig = \frac{k(k+1)}{2(k(k+1) -2s)}.
\end{align}
Clearly, the first integral in \eqref{H\"older} is $J_{k(k+1)/2,k}(2X)$. Moreover, from \eqref{g-def}, \eqref{poly} and \eqref{deg-p} it follows via an integer change of variables on $\balp$ together with the periodicity of the exponential sums that
\begin{align*}
	\int_{\T^k} |g_k(\balp;X)|^{2\sig}\d \balp = \int_{\T^{k-\ell}} |f_{k-\ell}(\bgam;X)|^{2\sig}\d \bgam = J_{\sig,k-\ell}(X).
\end{align*}
Consequently, upon inserting the bound \eqref{VMVT} from \cite[Theorem~1.1]{BDG}, we conclude that 
\begin{align*}
	H_{s,k}(X;\ba)	&\ll \left(J_{k(k+1)/2,k}(2X)\right)^{\frac{2s}{k(k+1)}} \left(J_{\sig,k-\ell}(X)\right)^{1/(2\sig)}\\
	&\ll X^{\eps} \left(X^{k(k+1)/2}\right)^{\frac{2s}{k(k+1)}} \left(X^\sig + X^{2\sig - (k-\ell)(k-\ell+1)/2}\right)^{1/(2\sig)}\\
	&\ll X^{s + \eps} \left(X^{1/2} + X^{1-(k-\ell)(k-\ell+1)/(4\sig)} \right).
\end{align*}
The proof of Theorem~\ref{MT} is now complete upon recalling Lemma~\ref{J<H} and inserting the value of $\sig$ from \eqref{sig}.

\section{Further discussion: Generalisations and limitations}
It is natural to ask in what ways the proof of Theorem~\ref{MT} may be generalised. 
As mentioned in the introduction, a close reading of our methods reveals that the translation-invariant structure of the system does indeed play a crucial role in our arguments. Consequently, there is little hope to extend our results to lacunary Vinogradov systems without fundamentally different ideas. 

However, the idea underpinning the proof of Theorem~\ref{MT} is quite general, and can be used to bound the number of solutions of the inhomogeneous analogues of other translation-invariant systems, including multidimensional ones such as those considered in \cite{GZk}. Indeed, beyond a sharp mean value estimate for such systems in the subcritical range, we only need a non-trivial bound for a derived mean value over a family of `shifting polynomials' analogous to the polynomials $p_j$ occurring in Lemma~\ref{J<H}. Owing to its genesis via shifting the variables, this derived mean value can be seen to be of the same general shape as the original one (but with lower degree), and therefore a similar mean value estimate can usually be applied, leading to a non-trivial bound as soon as at least one of the shifting polynomials is non-constant. Alternatively, one could also use square-root cancellation at the second moment (Plancherel's theorem) and interpolate with trivial bounds at $L^\infty$ to obtain cancellation at all moments greater than $2$. 
In this manner, our method gives a quantitative improvement for the number of solutions of inhomogeneous translation-invariant systems of equations over the corresponding bounds for their homogeneous analogues for all even moments in the subcritical range.  

A different direction for generalisations is that in which the variables $x_i, y_i$ are restricted to some subset $\calX \subseteq [-X,X]\cap \Z$. Fixing such a set $\calX$ as well as a set of shifting variables $\calH \subset [-X,X]\cap\Z$, put $J_{s,k}(\calX;\ba)$ for the number of $\bx, \by \in \calX^s$ satisfying \eqref{IV-sys}, and write $H_{s,k}(\calX, \calH;\ba)$ for the number of $\bx, \by \in (\calX+\calH)^s$ and $h \in \calH$ solving the system \eqref{s+a}. Then an inspection of the proof of Lemma \ref{J<H} shows that 
$$
	J_{s,k}(\calX;\ba) \ll |\calH|^{-1} H_{s,k}(\calX,\calH;\ba).
$$ 
Since moreover it follows from \cite[Theorem~1.2]{BDG} and \cite[Theorem~1.1]{W:NEC} that 
$$
	J_{s,k}(\calX;\bzer) \ll X^{\eps} (|\calX|^s + |\calX|^{2s-k(k+1)/2}),
$$ 
the same argument as above leads \emph{mutatis mutandis} to the following conclusion. 
\begin{theorem}\label{T-sets}
	Suppose that $k \ge 2$ and $\ba \in \Z^k \setminus \{\bzer\}$. Let $\ell\in \{1, \ldots, k\}$ be the smallest integer for which $a_{\ell} \neq 0$. Moreover, fix a set $\calX \subseteq [-X,X] \cap \Z$. Then for any integer $s < \frac12 k(k+1)$, any set $\calH \subseteq [-X,X]  \cap \Z$ and any $\eps > 0$ we have 
	\begin{align*}
		J_{s,k}(\calX; \ba)\ll X^\eps  |\calX+\calH|^{s}\left(|\calH|^{-1/2} + |\calH|^{- \eta_{s,k}(\ell) }\right),
	\end{align*} 
	where $\eta_{s,k}(\ell)$ is as in Theorem~\ref{MT}.
\end{theorem}

From the point of view of harmonic analysis, it would be of interest to have an analogue of Theorem~\ref{MT} involving systems of the type \eqref{IV-sys} in which each solution is counted by a complex weight. Such situations arise habitually in the context of restriction and extension operators. For a given complex-valued sequence $(\frc_n)_{n \in \Z}$ one defines the truncated extension operator along the moment curve $(t,t^2, \ldots, t^k) \subset \R^k$ via 
$$
	E_X\frc(\balp) = \sum_{|x| \le X} \frc_x e(\alp_1 x + \ldots + \alp_k x^k),
$$
so that $f_k(\balp;X) = E_X \mathbbm{1}(\balp)$.
In this notation, the conclusion of \cite[Theorem~1.2]{BDG} and \cite[Theorem~1.1]{W:NEC} reads 
$$
	\|E_X\frc \|_{L^{2s}(\T^k)} \ll X^{\eps} (1+X^{1/2-k(k+1)/(4s)})\|\frc\|_{\ell^2([-X,X])}.
$$

Consider the Fourier transform of $|E_X\frc(\balp)|^{2s}$, which is given by 
$$
	\Phi(\bn) = \int_{\T^k} |E_X \frc(\balp)|^{2s} e(-\balp \cdot \bn) \d \balp,
$$
so that in particular we have $	\|E_X \frc\|^{2s}_{L^{2s}(\T^k)} = \Phi(\bzer)$. Of course, it follows directly via the triangle inequality that $|\Phi(\bn)| \le \Phi(\bzer)$ for all $\bn \in \Z^k$, but it would be desirable to have a stronger inequality, and ideally one that exhibits a power saving over the trivial bound. In the special case when $\frc$ is the indicator function on a set $\calX \subseteq \Z \cap [-X,X]$, the results of Theorems \ref{MT} and \ref{T-sets} achieve such a power saving for  a large selection of $\bn$ and for all $s \in \N$ below the critical point. In fact, it seems not implausible that a careful modification of our arguments should apply also to a setting where the sequence $(\frc_x)_{x \in \Z}$ takes non-negative real values. Meanwhile, it is less obvious how the proof of Lemma \ref{J<H} would carry over to the general situation in which the weights $\frc$ may take complex values. 

Finally, we remark that the savings of $X^{1/2}$ in Theorem \ref{MT} is optimal within the methods. This can be seen by noting that the shifting and averaging operation creates one dummy variable, on which we can expect no more than square-root cancellation. The analytically inclined reader may also be interested to note that the results of \cite{BDG,W:NEC} require the underlying curve to have torsion, and this is what enables the strong bounds on mean values of the exponential sum $f_k(\balp;X)$. The curve underlying the exponential sum $g_k(\balp;X)$, however, given by $(p_1(t), \ldots, p_k(t))$, does not have torsion over $\R^k$, but it can be restricted to a $(k-\ell)$-dimensional subspace in which it has torsion. This gives a geometric motivation as to why our results are weaker as $\ell$ increases, and fail to give any improvement at all when $\ell$ is equal to $k$.

\bigskip 
\textbf{Acknowledgements.} During the production of this manuscript, the first author was supported by the Swedish Science Foundation (Vetenskapsr{\aa}det) via Starting Grant no.~2017-05110. The authors are grateful to Igor Shparlinski for useful discussions surrounding the topic of this paper.

\bibliographystyle{amsbracket}
\providecommand{\bysame}{\leavevmode\hbox to3em{\hrulefill}\thinspace}

\end{document}